\documentclass[12pt,a4paper]{amsart}

\usepackage[utf8]{inputenc}
\usepackage[T1]{fontenc}
\usepackage[ngerman, english]{babel}
\usepackage{lmodern}
\usepackage{times}

\usepackage{graphicx}
\usepackage{latexsym}
\usepackage{dsfont}
\usepackage{amssymb,amsmath,amsfonts,amsthm}
\usepackage{array}

\usepackage{tikz}
\usetikzlibrary{matrix,arrows}

\newtheorem{Satz}{Satz}

\newtheorem{corollary}[Satz]{Corollary}
\newtheorem{theorem}[Satz]{Theorem}

\newtheorem*{lemma*}{Lemma}

\theoremstyle{definition}

\newcommand{\C}{\mathbb{C}}

\newcommand{\R}{\mathbb{R}}

\newcommand{\N}{\mathbb{N}}
\newcommand{\T}{\mathbb{T}}

\newcommand{\cK}{\mathcal{K}}
\newcommand{\cB}{\mathcal{B}}

\newcommand{\cH}{\mathcal{H}}

\newcommand{\cP}{\mathcal{P}}
\newcommand{\cF}{\mathcal{F}}

\newcommand{\conv}{\mathrm{conv}}

\newcommand{\interior}{\mathrm{int \, }}
\newcommand{\relint}{\mathrm{relint \, }}

\newcommand{\intd}{\mathrm{d}}

\newcommand{\1}{\mathds{1}}

\DeclareMathOperator{\SO}{SO}

\DeclareMathOperator{\GOp}{G}
\DeclareMathOperator{\AOp}{A}

\newcommand{\Sn}{\mathbb S^{n-1}}

\newcommand{\CM}[1]
  {
    \phi_{#1}
  }

\newcommand{\TenCM}[4]
  {
    \phi_{#1}^{#2,#3,#4}
  }

\newcommand{\LokMinTen}[4]{{\phi_{#1}^{#2,#3,#4}}}

\begin{document}

  \title[Crofton formulae for tensorial curvature measures: the general case]
    {Crofton formulae for tensorial curvature measures: \\ the general case}
  \date{March 21, 2017}

  \author{Daniel Hug \and Jan A. Weis}
  \address{Karlsruhe Institute of Technology (KIT), Department of Mathematics, D-76128 Karls\-ruhe, Germany}
  \email{daniel.hug@kit.edu}
  \address{Karlsruhe Institute of Technology (KIT), Department of Mathematics, D-76128 Karls\-ruhe, Germany}
  \email{jan.weis@kit.edu}

  \thanks{The authors were supported in part by DFG grants FOR 1548 and HU 1874/4-2}

  \begin{abstract}
    The tensorial curvature measures are tensor-valued
    generalizations of the curvature measures of convex
    bodies. On  convex polytopes, there exist further generalizations
    some of which also have continuous extensions to arbitrary convex bodies.
    In a previous work, we obtained kinematic
    formulae for all (generalized) tensorial curvature measures.
    As a consequence of these results, we now derive a
	complete system of Crofton formulae for such
    (generalized) tensorial curvature measures.
    These formulae express the integral mean of the (generalized)
    tensorial curvature measures of the intersection of a
    given convex body (resp. polytope, or finite unions thereof) with a uniform
    affine $k$-flat in terms of linear combinations of (generalized)
    tensorial curvature measures of the given convex body
    (resp. polytope, or finite unions thereof).
    The considered generalized tensorial curvature measures
	generalize those studied formerly in the context of
	Crofton-type formulae, and the coefficients involved
	in these results are substantially less technical
	and structurally more transparent
    than in previous works.
    Finally, we prove that essentially all generalized
    tensorial curvature measures on convex polytopes are linearly
    independent. In particular, this implies that the
    Crofton formulae which we prove in this contribution
    cannot be simplified further.
  \end{abstract}

  \maketitle

\section{Introduction}
\label{sec:1}

    The \emph{classical Crofton formula} is a major
    result in integral geometry.
    It expresses the integral mean of the intrinsic
    volume of a convex body intersected with a uniform
    affine subspace of the underlying Euclidean space in
    terms of another intrinsic volume of this convex
    body.
    More precisely, for a \textit{convex body}
    $K \in \cK^{n}$ (a nonempty, compact, convex set) in
    the $n$-dimensional Euclidean space
    $\R^{n}$, $n \in \N$, the classical Crofton formula
    (see \cite[(4.59)]{Schneider14}) states that
    \begin{equation} \label{Form_Croft_Classic}
      \int_{A(n, k)} V_{j} (K \cap E) \, \mu_{k} (\intd E)
      = \alpha_{n j k} V_{n - k + j} (K),
    \end{equation}
    for $k \in \{ 0, \ldots, n \}$ and
    $j \in \{ 0, \ldots, k \}$, where $A(n, k)$ is the
    affine Grassmannian of $k$-flats in $\R^{n}$, on which
    $\mu_{k}$ denotes the motion invariant Haar measure,
    normalized as in \cite[p.~588]{SchnWeil08}, and
    $$\alpha_{n j k} = \frac {\Gamma(\frac {n - k + j + 1} 2)
		\Gamma(\frac {k + 1} 2)} {\Gamma(\frac {n + 1} 2) \Gamma(\frac {j + 1} 2)}
		$$ is expressed in terms of specific values of the Gamma function $\Gamma(\cdot)$
    (see \cite[Theorem 4.4.2]{Schneider14}).

    The functionals $V_{i} :\cK^n\rightarrow \R$, for
    $i\in\{0,\ldots,n\}$, appearing in
    \eqref{Form_Croft_Classic}, are the
    \textit{intrinsic volumes}, which occur as the
    coefficients of the monomials in the
    \textit{Steiner formula}
    \begin{equation} \label{Form_Steiner}
      \mathcal{H}^n (K + \epsilon B^{n})
      = \sum_{j = 0}^{n} \kappa_{n - j} V_{j} (K) \epsilon^{n - j},
    \end{equation}
    which holds for all convex bodies $K \in \cK^{n}$ and
    $\epsilon \geq 0$.
    Here, $\mathcal{H}^n$ is the $n$-dimensional Hausdorff
    measure (Lebesgue measure, volume), $+$ denotes the
    Minkowski addition in $\R^{n}$, and $\kappa_{n}$ is
    the volume of the Euclidean unit ball $B^{n}$  in
    $\R^{n}$.
    Properties of the intrinsic volume $V_i$ such as
    continuity, isometry invariance, homogeneity, and
    additivity (valuation property) are derived from
    corresponding properties of the volume functional.
    A key result for the intrinsic volumes is
    \emph{Hadwiger's characterization theorem} (see
    \cite[2. Satz]{Hadwiger52}), which states that
    $V_{0}, \ldots, V_{n}$ form a basis of the vector
    space of continuous and isometry invariant real-valued
    valuations on $\cK^{n}$.
    One of its numerous applications is a concise proof of
    \eqref{Form_Croft_Classic}.

    A natural way to extend the classical Crofton formula
    is to apply the integration over the affine
    Grassmannian $A(n, k)$ to functionals which generalize
    the intrinsic volumes.
    One of these generalizations concerns tensor-valued
    valuations on $\cK^{n}$.
    Their systematic investigation started with a
    characterization theorem, similar to the
    aforementioned result due to Hadwiger.
    Integral geometric formulae, including a
    Crofton formula, for \textit{quermassvectors}
    (vector-valued generalizations of the intrinsic volumes)
    have already been  found by Hadwiger \& Schneider and
    Schneider, in 1971/72 (see
    \cite{HadSchn71, Schn72, Schn72b}).
    More recently in 1997, McMullen generalized these
    vector-valued valuations even further, and introduced
    tensor-valued generalizations of the intrinsic volumes
    (see \cite{McMullen97}).
    Only two years later Alesker generalized Hadwiger's
    characterization theorem (see
    \cite[Theorem 2.2]{Alesker99b}) by showing that the
    vector space of continuous and isometry covariant
    tensor-valued valuations on $\cK^{n}$ is spanned by the
    tensor-valued versions of
    the intrinsic volumes, the \emph{Minkowski tensors},
    multiplied with suitable powers of the metric tensor
    in~$\R^{n}$.
    However, these valuations are not linearly independent,
    as shown by McMullen (see \cite{McMullen97}) and further
    investigated by Hug, Schneider \& Schuster (see
    \cite{HSS07a}).
    This is one reason why an approach to explicit integral
    geometric formulae via characterization theorems does
    not seem to be technically feasible.
    Nevertheless, great progress in the integral geometry
		of tensor-valued valuations has been
    made by different methods.
    In 2008, Hug, Schneider \& Schuster proved a set of
    Crofton formulae for the Minkowski tensors (see
    \cite[Theorem 2.1--2.6]{HugSchnSchu08}).
    A totally different algebraic approach has been
    developed by Bernig \& Hug to obtain    various
    integral geometric formulae for the translation
    invariant Minkowski tensors (see \cite{BernHug15}).

    On the other hand, localizations of the intrinsic volumes
    yield other types of generalizations.
    The \emph{support measures} are weakly continuous,
    locally defined and motion equivariant valuations on
    convex bodies with values in the space of finite measures
    on Borel subsets of    $\R^{n} \times \Sn$, where $\Sn$
    denotes the Euclidean unit sphere in $\R^{n}$.
    They are determined by a local version of
    \eqref{Form_Steiner} and form a crucial example of
    localizations of the intrinsic volumes, which are simply
    the total support measures.
    Furthermore, their marginal measures on Borel subsets of
    $\R^{n}$ are called \emph{curvature measures}, and the
    ones on Borel subsets of $\Sn$ are called \emph{area measures}.
    For the area measures and the curvature measures, Schneider found
    characterization theorems (see
    \cite{Schneider75, Schneider78}) similar to the one due
    to Hadwiger in the global case.
    It took some time until in 1995, Glasauer proved a
    characterization theorem for the support measures, even
    without the need of requesting the valuation property
    (see \cite[Satz 4.2.1]{Glasauer95}).
    As to integral geometry, in 1959 Federer \cite{Federer59}
    proved Crofton formulae for    curvature measures, even
    in the more general setting of sets with positive reach.
    Certain Crofton formulae for support measures were proved
    by Glasauer in 1997 (see \cite[Theorem~3.2]{Glasauer97}).
    However, his results require a special set operation on
    support elements of the involved convex bodies and affine
    subspaces.

    Interestingly, the combination of Minkowski tensors and
    localization leads to a better understanding of integral
    geometric formulae.
    In recent years, Schneider introduced \emph{local tensor
    valuations} (see \cite{Schneider13}), which were then
    further studied by Hug \& Schneider (see
    \cite{HugSchn14, HugSchn16a, HugSchn16b}).
    They introduced particular tensor-valued support measures,
    the \emph{local Minkowski tensors} on convex bodies
    (and generalizations on polytopes), which (as their
    name suggests) can also be seen as localizations of the
    Minkowski tensors.
    They proved several different characterization results for
    these (generalized) local Minkowski tensors in the just
    mentioned works.
    This led us to consider their marginal measures on Borel
    subsets of $\R^{n}$, the \emph{tensorial curvature measures}
    and their generalizations on convex polytopes.
    Preceding this work, the present authors derived a set of
    Crofton formulae for a different version of these tensorial
    curvature measures, defined with respect to the (random)
    intersecting affine subspace, and as a consequence of these
    results also obtained Crofton formulae for some of the
    (original) tensorial curvature measures (see
    \cite{HugWeis16a}).
    As a far reaching generalization of previous results, a
    complete set of kinematic formulae for the (generalized)
    tensorial curvature measures has been proved in
    \cite{HugWeis16b}.

    In statistical physics, the intrinsic volumes (in the physical
		context better known as \textit{Minkowski functionals}) are an
		important tool for the characterization of geometric properties of spatial
		patterns (see for example the survey \cite{Mecke00}).
    However, due to their translation and rotation invariance, they are not useful
		when it comes to the quantification of orientation or anisotropy of spatial structures.
    For the determination of these kinds of geometric features,
		rotation covariant valuations, such as the Minkowski tensors and the (generalized)
		tensorial curvature measures considered here, are much more suitable, and
		therefore they have been
		heavily used recently (see for example \cite{Schroeder11}). For further
		applications we refer to the introduction of our preceding work~\cite{HugWeis16b}.

    The aim of the present work is to prove a complete set of
    Crofton formulae for the (generalized) tensorial curvature
    measures.
    This complements the particular results for (extrinsic)
    tensorial curvature measures and Minkowski tensors obtained
    in \cite{HugWeis16a} and \cite{SvaneJensen16}.
    The current approach is basically an application of the
    kinematic formulae for (generalized) tensorial curvature
    measures derived in \cite{HugWeis16b}.
    The connection between local kinematic and local Crofton
    formulae is well known for the scalar curvature measures.
    In that case it is used to determine the coefficients in
    the kinematic formulae.
    The basic strategy there is as follows.
    First, the kinematic formulae are proved, but the involved
    coefficients remain undetermined, since the required direct
    calculation seemed to be infeasible.
    Then Crofton formulae are derived which involve the same
    constants.
    In the latter, the determination of the coefficients turns
    out to be an easy task, which is accomplished by evaluating
    the result for balls of different radii.
    In the tensorial framework, this  approach breaks down,
    since the explicit calculation of integral mean values of (generalized)
    tensorial curvature measures for sufficiently many examples
    (template method) does not seem to be possible.
    Instead, the required coefficients were determined by a
    direct derivation of the kinematic formulae for (generalized) tensorial
    curvature measures in \cite{HugWeis16b}, and from this we now
    can derive explicit Crofton formulae, otherwise following
    the  reasoning described above.

    Since the tensorial curvature measures are local versions of the Minkowski tensors, it is rather straightforward to derive Crofton formulae, similar to the ones proved in the present contribution, and apparently also kinematic formulae, similar to the ones obtained in the preceding work \cite{HugWeis16b}, for Minkowski tensors as well.
    This is the subject of the subsequent work \cite{HugWeis17}.
    There we extend some of the integral formulae for translation invariant Minkowski tensors obtained by Bernig \& Hug in \cite{BernHug15} and significantly simplify the coefficients of the Crofton formulae proven in \cite{HugSchnSchu08}.
    We further refer to \cite{HugWeis16b} for a thorough
    discussion of related work.

    The present contribution is structured as follows. In Section
    \ref{sec:2}, we fix our notation and collect various
    auxiliary results which will be needed.
    Section \ref{sec:3} contains the main results.
    First, we state the Crofton formulae for the
    generalized tensorial curvature measures on the space
    $\cP^{n}$ of convex polytopes  in $\R^n$.
    Then we  provide the formulae for all the (generalized) tensorial
    curvature measures, for which a continuous
    extension to $\cK^{n}$ exists.
    Finally, we highlight some special cases.
    In Section~\ref{sec:4}, we first recall the kinematic
    formulae for generalized tensorial curvature measures from \cite{HugWeis16b},
		in order to apply these in the proofs of the main
    results and the corollaries. In the final Section \ref{sec:5}
    we show that the generalized tensorial curvature measures on convex polytopes are essentially
    all linearly independent.

\section{Preliminaries}
\label{sec:2}

    We work in the $n$-dimensional Euclidean space
    $\R^{n}$, equipped with its usual topology generated by the standard
    scalar product $\langle \cdot\,, \cdot \rangle$ and the corresponding Euclidean norm
    $\| \cdot \|$. For a topological space $X$, we denote the Borel
    $\sigma$-algebra on $X$ by $\cB(X)$.

    We denote the proper (orientation preserving)
		rotation group on $\R^{n}$ by $\SO(n)$, and we write $\nu$ for the
    Haar probability measure on $\SO(n)$. By $\GOp(n, k)$ (resp.~$\AOp(n, k)$),
    for $k \in \{0, \ldots, n\}$, we denote the
    Grassmannian of $k$-dimensional linear (resp.~affine) subspaces of $\R^{n}$.
		These sets carry natural topologies, see \cite[Chap.~13.2]{SchnWeil08}.
		
		We
    write $\mu_k$  for the rotation  invariant Haar  measure
    on $\AOp(n, k)$, normalized as in \cite[(13.2)]{SchnWeil08}, that is, for a fixed (but arbitrary) linear subspace $E_k \in \GOp(n, k)$,
		\begin{equation}\label{decomp}
		\mu_k(\cdot)=\int_{\SO(n)}\int_{E_k^\perp}\1\{\rho(E_k+t)\in\cdot\}\,
		\mathcal{H}^{n - k}(\intd t)\, \nu(\intd \rho),
		\end{equation}
		where $\cH^{j}$ denotes the $j$-dimensional Hausdorff measure for $j \in \{0, \ldots, n\}$.
		The directional space of an affine $k$-flat $E \in
    \AOp(n, k)$ is denoted by $E^{0} \in \GOp(n, k)$ and its orthogonal complement
    by $E^{\perp} \in \GOp(n, n - k)$.
    The orthogonal projection of a vector $x \in
    \R^{n}$ to a linear subspace $L$ of $\R^{n}$ is denoted by
    $p_{L}(x)$.

    The vector space of symmetric tensors of rank $p \in \N_{0}$ over
    $\R^n$ is denoted by~$\T^{p}$, and the corresponding algebra of symmetric
    tensors over $\R^{n}$ by $\T$. The symmetric tensor product of two
    tensors $T, U \in \T$ is denoted by $TU$, and for $q \in \N_{0}$ and a tensor
    $T \in \T$ we write $T^{q}$ for the $q$-fold tensor product; see also the
    introductory chapters of \cite{KidVed16} for further details and references.
    Identifying $\R^{n}$
    with its dual space via its scalar product, we interpret a symmetric
    tensor of rank $p$ as a symmetric $p$-linear map from
    $(\R^{n})^{p}$ to $\R$. One special tensor is the \emph{metric
    tensor} $Q \in \T^{2}$, defined by $Q(x, y) := \langle x, y \rangle$ for $x,
    y \in \R^{n}$. For an affine $k$-flat $E \in \AOp(n, k)$, $k \in \{0,
    \ldots, n\}$, the metric tensor $Q(E)$ associated with $E$ is defined by $Q(E)(x,
    y) := \langle p_{E^{0}} (x), p_{E^{0}} (y) \rangle$ for $x, y \in \R^{n}$.

    The definition of the (generalized) tensorial curvature measures, which we
		recall in the following, is partly motivated
    by their relation to the
    support measures. Therefore, we first recall the latter.
    For a convex body $K \in \cK^{n}$ and
    $x \in \R^{n}$, we denote the
    metric projection of $x$ onto $K$ by $p(K, x)$, and for $x \in \R^{n} \setminus
    K$ we define $u(K, x)
    := (x - p(K, x)) / \| x - p(K, x) \|$. For $\epsilon >
    0$ and a Borel set $\eta \subset \Sigma^{n}:=\R^{n} \times
    \Sn$,
    \begin{equation*}
      M_{\epsilon}(K, \eta) := \left\{ x \in \left( K + \epsilon B^{n} \right)
      \setminus K \colon \left( p(K, x), u(K, x) \right) \in \eta \right\}
    \end{equation*}
    is a local parallel set of $K$ which satisfies the \emph{local Steiner
      formula}
    \begin{equation} \label{Form_Steiner_loc}
      \mathcal{H}^n(M_{\epsilon}(K, \eta))
      = \sum_{j = 0}^{n - 1} \kappa_{n - j} \Lambda_{j} (K, \eta)
      \epsilon^{n - j}, \qquad \epsilon \geq 0.
    \end{equation}
    This relation determines the \emph{support measures} $\Lambda_{0}
    (K, \cdot), \ldots, \Lambda_{n - 1} (K, \cdot)$ of $K$, which are
    finite, nonnegative Borel measures on $\cB (\Sigma^{n})$. Obviously, a comparison
    of \eqref{Form_Steiner_loc} and the Steiner formula \eqref{Form_Steiner} yields
    $V_{j}(K) = \Lambda_{j} (K, \Sigma^{n})$. Further information
    on these measures and functionals can be found in \cite[Chap.~4.2]{Schneider14}.

    Let $\cP^n\subset\cK^n$ denote the space of convex polytopes in $\R^n$.
    For a polytope $P \in \cP^{n}$ and $j \in \{ 0, \ldots, n \}$, we
    denote the set of $j$-dimensional faces of $P$ by $\cF_{j}(P)$ and
    the normal cone of $P$ at a face $F \in \cF_{j}(P)$ by $N(P,F)$.
    Then, the $j$th support measure $\Lambda_{j} (P, \cdot)$ of $P$ is explicitly given by
    \begin{equation*}
      \Lambda_{j} (P, \eta) = \frac {1} {\omega_{n - j}}
      \sum_{F \in \cF_{j}(P)} \int _{F} \int _{N(P,F) \cap \Sn}
      \1_\eta(x,u)
      \, \cH^{n - j - 1} (\intd u) \, \cH^{j} (\intd x)
    \end{equation*}
    for $\eta \in \cB(\Sigma^{n})$ and $j \in \{ 0, \ldots, n - 1 \}$,
    where $\omega_{n}$ is the $(n - 1)$-dimensional volume (Hausdorff measure) of $\Sn$.

    For a polytope $P\in\mathcal{P}^n$, we define the
    \emph{generalized tensorial curvature measure}
    \begin{align*}
      \TenCM{j}{r}{s}{l} (P, \cdot) ,\qquad j \in\{0, \ldots, n - 1\}, \, r, s, l \in \N_{0},
    \end{align*}
    as the Borel measure on $\mathcal{B}(\R^n)$ which is given by
    \begin{equation*}
      \TenCM{j}{r}{s}{l} (P, \beta) :=  c_{n, j}^{r, s, l}\,\frac{1}{\omega_{n - j}}
			\sum_{F \in \cF_{j}(P)} Q(F)^{l} \int _{F \cap \beta} x^r \, \cH^{j}(\intd x)
			\int _{N(P,F) \cap \Sn} u^{s} \, \cH^{n - j - 1} (\intd u),
    \end{equation*}
    for $\beta \in \cB(\R^{n})$, where
    \begin{align*}
      c_{n, j}^{r, s, l} := \frac{1}{r! s!} \frac{\omega_{n - j}}{\omega_{n - j + s}} \frac{\omega_{j + 2l}}{\omega_{j}}  \text{ if }j\neq 0, \ \text{ $c_{n, 0}^{r, s, 0} := \frac{1}{r! s!}
      \frac{\omega_{n}}{\omega_{n + s}}$},\ \text{ and $c_{n, 0}^{r, s, l} :=1$ for $l\ge 1$.}
    \end{align*}
    Note that if $j = 0$ and $l \ge 1$, then we have $\TenCM{0}{r}{s}{l} \equiv 0$.
    In all other cases the factor $1/\omega_{n-j}$ in the definition of $\TenCM{j}{r}{s}{l} (P, \beta)$ and the factor $\omega_{n-j}$ involved in the constant $c_{n, j}^{r, s, l}$ cancel.
		
    For a general convex body $K\in\mathcal{K}^n$, we define the \emph{tensorial curvature measure}
    \begin{align*}
      \TenCM{n}{r}{0}{l} (K, \cdot),\qquad r,l\in\N_0,
    \end{align*}
    as the Borel measure on $\mathcal{B}(\R^n)$ which is given by
    \begin{equation*}
      \TenCM{n}{r}{0}{l} (K, \beta): = c_{n, n}^{r, 0, l} \, Q^{l} \int _{K \cap \beta} x^r \, \cH^{n}(\intd x),
    \end{equation*}
    for $\beta \in \cB(\R^{n})$,
    where  $c_{n, n}^{r, 0, l} := \frac{1}{r!} \frac{\omega_{n + 2l}}{\omega_{n}}$.
    For the sake of convenience, we extend these definitions by $\TenCM{j}{r}{s}{0} := 0$
    for $j \notin \{ 0, \ldots, n \}$ or $r \notin \N_{0}$ or $s \notin \N_{0}$ or $j = n$ and $s \neq 0$. Finally, we
    observe that for $P\in\mathcal{P}^n$, $r=s=l=0$, and $j=0,\ldots,n-1$, the scalar-valued measures
    $\TenCM{j}{0}{0}{0}(P,\cdot)$ are just the curvature measures $\CM{j}(P,\cdot)$, that is,
    the marginal measures on $\R^n$ of the support measures $\Lambda_j(P,\cdot)$, which therefore can be extended from
    polytopes to general convex bodies, and $\TenCM{n}{0}{0}{0}(K,\cdot)$ is the restriction
    of the $n$-dimensional Hausdorff measure to $K\in\mathcal{K}^n$.

    We emphasize that in the present work,
    $\TenCM{j}{r}{s}{l} (P, \cdot)$ and
    $\TenCM{n}{r}{0}{l} (K, \cdot)$ are Borel measures on $\R^n$
    and not on $\R^n\times\mathbb{S}^{n-1}$, as in
    \cite{HugSchn14}, and also the normalization is slightly
    adjusted as compared to \cite{HugSchn14}
    (where the normalization was not a relevant issue).
    However, we stick to the definition and normalization of our preceding work  \cite{HugWeis16b},
    where the connection to the generalized local Minkowski tensors $\tilde{\phi}^{r,s,l}_j$
    is described and where also the properties and available
		characterization results for these measures are discussed in more detail.

    It has been shown in \cite{HugSchn14} that the generalized local Minkowski tensor $\tilde{\phi}^{r,s,l}_j$ has a continuous extension
    to $\cK^{n}$ which preserves all other properties if and only if $l \in \{0,1\}$;
    see \cite[Theorem 2.3]{HugSchn14} for a  stronger characterization result.
    Globalizing any such continuous extension in the $\Sn$-coordinate,
    we obtain a continuous extension for the generalized tensorial curvature measures.
    These can be represented with suitable differential forms which are defined on the sphere bundle
		of $\R^n$ and integrated (that is, evaluated) on the normal cycle, which is the reason why
		they are called \textit{smooth} (for more details, see for example
		\cite{Saienko16}, \cite{HugSchn14}, \cite{KidVed16} and
		the literature cited there).
    For $l = 0$,  this extension can be easily expressed
    via the support measures. We call the measures thus obtained the
    \emph{tensorial curvature measures}. For a convex body $K \in \cK^{n}$, a Borel
    set $\beta \in \cB(\R^{n})$, and $r, s \in \N_{0}$, they are given by
    \begin{equation}\label{alsomeas}
      \TenCM{j}{r}{s}{0} (K, \beta):= c_{n, j}^{r, s, 0} \int _{\beta \times \Sn} x^r u^s \, \Lambda_j(K, \intd (x, u)),
    \end{equation}
    for $j\in\{0,\ldots,n-1\}$, whereas $ \TenCM{n}{r}{0}{l} (K, \beta)$ has
    already been defined for all $K \in \cK^{n}$.

    The valuations $Q^{m} \TenCM{j}{r}{s}{l}$ on $\cP^{n}$ are linearly independent, where $ m, j, r, s, l \in \N_{0}$,
    $j \in \{ 0, \ldots, n \}$ with $l = 0$, if $j \in \{ 0, n - 1 \}$, and $s = l = 0$, if $j = n$.
    In our preceding work \cite{HugWeis16b}, we pointed out that the corresponding proof is similar to the proof of \cite[Theorem~3.1]{HugSchn14}.
    A detailed argument is now provided in Section \ref{sec:5}.

    The statements of our results involve  the classical \emph{Gamma function}.
    For all complex numbers $z \in \C \setminus \{ 0, -1, \ldots \}$ (see \cite[(2.7)]{Artin64}), it can be defined via the Gaussian product formula
    \begin{align*}
      \Gamma(z) := \lim_{a \rightarrow \infty} \frac{a^{z} a!}{z (z + 1) \cdots (z + a)}.
    \end{align*}
    This definition implies, for $c \in \R\setminus \mathbb{Z}$ and $m\in \N_{0}$, that
	\begin{align} \label{Form_Gam_Cont}
      \frac{\Gamma(-c + m)}{\Gamma(-c)} = (-1)^{m} \frac{\Gamma(c + 1)}{\Gamma(c - m + 1)}.
    \end{align}
    The Gamma function has simple poles at the nonpositive integers.
    The expression on the right  side of relation \eqref{Form_Gam_Cont} provides a continuation of the expresion on the left side at  $c \in \N_{0}$, with
		the understanding that $\Gamma(c - m + 1)^{-1} = 0$ for $c < m$.
    This continuation will be applied several times (sometimes without mentioning it specifically) in the coefficients of the upcoming formulae and the corresponding proofs.
    Mostly it is used in ``boundary cases'' in which quotients are involved, such as in \eqref{Form_Gam_Cont} with $c = 0$ and $m \in \N_{0}$, which have to be interpreted as ${\Gamma(0 + m)}/{\Gamma(0)} = \1\{ m = 0 \}$.

\section{The Crofton formulae}
\label{sec:3}

    In this work, we establish a complete set of Crofton formulae for the generalized tensorial curvature measures of
		convex polytopes. That is, for $P \in \cP^n$ and $ \beta \in \cB(\R^n)$, we explicitly express integrals of the form
    \begin{align*}
      \int_{\AOp(n, k)} \TenCM{j}{r}{s}{l} (P \cap E, \beta \cap E) \, \mu_k( \intd E)
    \end{align*}
    in terms of generalized tensorial curvature measures of $P$, evaluated at $\beta$. Furthermore, for $l = 0, 1$, the corresponding (generalized) tensorial curvature measures are defined on $\cK^{n} \times \cB(\R^{n})$, and therefore we also consider the Crofton integrals
    \begin{align*}
      \int_{\AOp(n, k)} \TenCM{j}{r}{s}{l} (K \cap E, \beta \cap E) \, \mu_k( \intd E)
    \end{align*}
    for $K \in \cK^n$, $\beta \in \cB(\R^n)$, $l = 0, 1$.

    All results which are stated in the following, extend by additivity to finite unions of polytopes or convex bodies.

\subsection{Generalized tensorial curvature measures on polytopes}

   First, we separately state a formula for $j = k$.

    \begin{theorem} \label{Thm_CF_j=k}
      Let $P \in \cP^n$, $\beta \in \cB(\R^n)$, and $k, r, s, l \in \N_{0}$ with $k \leq n$. Then,
      \begin{align*}
        \int_{\AOp(n, k)} \TenCM{k}{r}{s}{l} (P \cap E, \beta \cap E) \,
				\mu_k( \intd E) = \1 \{ s \text{ \rm even} \} \frac{1}{(2\pi)^{s}
				\left(\frac{s}{2}\right)!} \frac{\Gamma(\frac{n - k + s}{2})}{
				\Gamma(\frac{n - k}{2})} \, \TenCM{n}{r}{0}{\frac s 2 + l}(P, \beta).
      \end{align*}
    \end{theorem}

    Theorem \ref{Thm_CF_j=k} generalizes Theorem 2.1 in \cite{HugSchnSchu08}. In fact, setting $l = 0$ and $\beta = \R^{n}$ one obtains the known result for Minkowski tensors. If $l \in \{0, 1\}$, one can even formulate Theorem \ref{Thm_CF_j=k} for a convex body, as in both of these cases all appearing valuations are defined on $\cK^{n}$. For $k = n$, the integral on the left-hand side of the formula in Theorem \ref{Thm_CF_j=k} is trivial. However, note that on the right-hand side the quotient of the Gamma functions has to be interpreted as $\1\{ s = 0 \}$, according to \eqref{Form_Gam_Cont}.

    Next, we state the formulae for general $j < k$.

    \begin{theorem} \label{Thm_CF_j<k}
      Let $P \in \cP^n$, $\beta \in \cB(\R^n)$, and $j, k, r, s, l \in \N_{0}$ with $j < k \leq n$, and with $l = 0$ if $j = 0$. Then,
      \begin{align*}
        \int_{\AOp(n, k)} \TenCM{j}{r}{s}{l} (P \cap E, \beta \cap E) \, \mu_k( \intd E) = \sum_{m = 0}^{\lfloor \frac s 2 \rfloor} \sum_{i = 0}^{m} d_{n, j, k}^{s, l, i, m} \, Q^{m - i} \TenCM{n - k + j}{r}{s - 2m}{l + i} (P, \beta),
      \end{align*}
      where
      \begin{align*}
        d_{n, j, k}^{s, l, i, m} & : = \frac{(-1)^{i}}{(4\pi)^{m} m!} \frac{\binom{m}{i}}{\pi^i} \frac {(i + l - 2)!} {(l - 2)!} \frac {\Gamma(\frac {n - k + j + 1} 2) \Gamma(\frac {k + 1} 2)} {\Gamma(\frac {n + 1} 2) \Gamma(\frac {j + 1} 2)} \\
        & \qquad \times \frac {\Gamma(\frac {n - k + j} 2 + 1)} {\Gamma(\frac {n - k + j + s} 2 + 1)} \frac {\Gamma(\frac {j + s} 2 - m + 1)} {\Gamma(\frac {j} 2 + 1)} \frac {\Gamma(\frac {n - k} 2 + m)} {\Gamma(\frac{n - k}{2})}.
      \end{align*}
    \end{theorem}

    For $k = n$ the coefficient in Theorem \ref{Thm_CF_j<k} has to be interpreted as
    \begin{align*}
      d_{n, j, n}^{s, l, i, m} & = \1\{ i = m = 0 \},
    \end{align*}
    according to \eqref{Form_Gam_Cont}, so that the result is a tautology in this case.
		
    Several remarkable facts concerning the coefficients $ d_{n, j, k}^{s, l, i, m}$ should be recalled from \cite{HugWeis16b}.
    First, the ratio ${(i + l - 2)!} /{(l - 2)!}$ has to be interpreted in terms of Gamma functions and relation \eqref{Form_Gam_Cont} if $l\in\{0,1\}$.
    The corresponding special cases will be considered separately in the following two theorems and the subsequent corollaries.
		Second, the coefficients are independent of the tensorial parameter r and, due to our normalization of the generalized tensorial curvature measures, depend only on $l$ through the ratio ${(i + l - 2)!} /{(l - 2)!}$.
    Third, only tensors $\TenCM{n-k+j}{r}{s - 2m}{p} (P, \beta) $ with $p \ge l$ show up on the right side of the kinematic formula.
    Using Legendre's duplication formula, we could shorten the given expressions for the  coefficients $ d_{n, j, k}^{s, l, i, m}$ even further.
    However, the present form has the advantage of exhibiting that the factors in the second line cancel each other if $s=0$ (and hence also $m=i=0$).
    Furthermore, in general the coefficients are signed in contrast to the classical kinematic formula.
    We shall see below that for $l\in\{0,1\}$ all coefficients are nonnegative.

  \subsection{(Generalized) tensorial curvature measures on convex bodies}

   The generalized tensorial curvature measures $\TenCM{j}{r}{s}{l}$ can be continuously extended to all convex bodies if $l \in \{0, 1\}$. In these two cases, Theorem \ref{Thm_CF_j=k} holds for general convex bodies as well. For this reason, we restrict our attention to the cases where $j < k$ in the following. The next theorems are stated without a proof, as they basically follow from Theorem \ref{Thm_CF_j<k} and approximation of the given convex body by polytopes (using the weak continuity of the corresponding generalized tensorial curvature measures and the usual arguments needed to take care of exceptional positions).

    We start with the formula for $l = 1$.

    \begin{theorem} \label{Thm_CF_j<k_l=1}
      Let $K \in \cK^n$, $\beta \in \cB(\R^n)$, and $j, k, r, s \in \N_{0}$ with $0 < j < k \leq n$. Then,
      \begin{align*}
        \int_{\AOp(n, k)} \TenCM{j}{r}{s}{1} (K \cap E, \beta \cap E) \, \mu_k( \intd E) = \sum_{m = 0}^{\lfloor \frac s 2 \rfloor} d_{n, j, k}^{s, 1, 0, m} \, Q^{m} \TenCM{n - k + j}{r}{s - 2m}{1} (K, \beta),
      \end{align*}
      where
      \begin{align*}
        d_{n, j, k}^{s, 1, 0, m} & = \frac{1}{(4\pi)^{m} m!} \frac {\Gamma(\frac {n - k + j + 1} 2) \Gamma(\frac {k + 1} 2)} {\Gamma(\frac {n + 1} 2) \Gamma(\frac {j + 1} 2)} \\
        & \qquad \times \frac {\Gamma(\frac {n - k + j} 2 + 1)} {\Gamma(\frac {n - k + j + s} 2 + 1)} \frac {\Gamma(\frac {j + s} 2 - m + 1)} {\Gamma(\frac {j} 2 + 1)} \frac {\Gamma(\frac {n - k} 2 + m)} {\Gamma(\frac{n - k}{2})}.
      \end{align*}
    \end{theorem}

    Next, we state the formula for $l = 0$.

    \begin{theorem} \label{Thm_CF_j<k_l=0}
      Let $K \in \cK^n$, $\beta \in \cB(\R^n)$ and $j, k, r, s \in \N_{0}$ with $j < k \leq n$. Then,
      \begin{align*}
        \int_{\AOp(n, k)} \TenCM{j}{r}{s}{0} (K \cap E, \beta \cap E) \, \mu_k( \intd E) = \sum_{m = 0}^{\lfloor \frac s 2 \rfloor} \sum_{i = 0}^{1} d_{n, j, k}^{s, 0, i, m} \, Q^{m - i} \TenCM{n - k + j}{r}{s - 2m}{i} (K, \beta),
      \end{align*}
      where
      \begin{align*}
        d_{n, j, k}^{s, 0, i, m} & = \frac{1}{(4\pi)^{m} m!} \frac{\binom{m}{i}}{\pi^i} \frac {\Gamma(\frac {n - k + j + 1} 2) \Gamma(\frac {k + 1} 2)} {\Gamma(\frac {n + 1} 2) \Gamma(\frac {j + 1} 2)} \\
         & \qquad \times \frac {\Gamma(\frac {n - k + j} 2 + 1)} {\Gamma(\frac {n - k + j + s} 2 + 1)} \frac {\Gamma(\frac {j + s} 2 - m + 1)} {\Gamma(\frac {j} 2 + 1)} \frac {\Gamma(\frac {n - k} 2 + m)} {\Gamma(\frac{n - k}{2})}.
      \end{align*}
    \end{theorem}

    In Theorem \ref{Thm_CF_j<k_l=0}, we have $d_{n, j, k}^{s, 0, 1, 0} = 0$ so that in fact the undefined tensor $Q^{-1}$ does not appear.

    For the special case $j = k - 1$, we deduce two more Crofton formulae. The first concerns the generalized tensorial curvature measures $\TenCM{k - 1}{r}{s}{1}$.

    \begin{corollary} \label{Cor_CF_j=k-1_l=1}
      Let $K \in \cK^n$, $\beta \in \cB(\R^n)$, and $k, r, s \in \N_{0}$ with $0 < k < n$. Then,
      \begin{align*}
        \int_{\AOp(n, k)} \TenCM{k - 1}{r}{s}{1} (K \cap E, \beta \cap E) \, \mu_k( \intd E) = \sum_{m = 0}^{\lfloor \frac s 2 \rfloor} \iota_{n, k}^{s, m} \, Q^{m} \TenCM{n - 1}{r}{s - 2m}{1} (K, \beta),
      \end{align*}
      where
      \begin{align*}
        \iota_{n, k}^{s, m} & : = \frac{1}{(4\pi)^{m} m!} \frac {\Gamma(\frac {n} 2) \Gamma(\frac {k + s + 1} 2 - m) \Gamma(\frac {n - k} 2 + m)} {\Gamma(\frac {n + s + 1} 2) \Gamma(\frac {k} 2) \Gamma(\frac{n - k}{2})}.
      \end{align*}
    \end{corollary}

    Due to the easily verified relation
    \begin{align}
      \TenCM{n - 1}{r}{s - 2m}{1} = \frac{2\pi}{n - 1} \left( Q \TenCM{n - 1}{r}{s - 2m}{0} - 2\pi(s - 2m + 2) \TenCM{n - 1}{r}{s - 2m + 2}{0} \right) , \label{FormTenCMRel}
    \end{align}
    Corollary \ref{Cor_CF_j=k-1_l=1} can be transformed in such a way that only the tensorial curvature measures $\TenCM{n - 1}{r}{s - 2m}{0}$ are involved on the right-hand side of the preceding formula. This is presented in the following corollary.

    \begin{corollary} \label{Cor_CF_j=k-1_l=1Alt}
      Let $K \in \cK^n$, $\beta \in \cB(\R^n)$, and $k, r, s \in \N_{0}$ with $1 < k < n$. Then,
      \begin{align*}
        \int_{\AOp(n, k)} \TenCM{k - 1}{r}{s}{1} (K \cap E, \beta \cap E) \, \mu_k( \intd E) = \sum_{m = 0}^{\lfloor \frac s 2 \rfloor + 1} \lambda_{n, k}^{s, m} \, Q^{m} \TenCM{n - 1}{r}{s - 2m + 2}{0} (K, \beta),
      \end{align*}
      where
      \begin{align*}
        \lambda_{n, k}^{s, m} & := \frac{\pi}{(n - 1)(4\pi)^{m - 1} m!} \frac {\Gamma(\frac {n} 2)\Gamma(\frac {k + s + 1} 2 - m) \Gamma(\tfrac {n - k} 2 + m - 1)} {\Gamma(\frac {n + s + 1} 2) \Gamma(\frac {k} 2) \Gamma(\frac{n - k}{2})} \\
        & \qquad \times \left( 2 m (\tfrac {k + s + 1} 2 - m) - (s - 2m + 2) (\tfrac {n - k} 2 + m - 1) \right),
        \intertext{for $m \in \{ 1, \ldots, \lfloor \frac s 2 \rfloor \}$, and}
        \lambda_{n, k}^{s, 0} & := - \frac{4\pi^{2}(s + 2)}{n - 1} \frac {\Gamma(\frac {n} 2) \Gamma(\frac {k + s + 1} 2)} {\Gamma(\frac {n + s + 1} 2) \Gamma(\frac {k} 2)}, \\
        \lambda_{n, k}^{s, \lfloor \frac s 2 \rfloor + 1}& := \frac{2 \pi}{(n - 1)(4\pi)^{\lfloor \frac s 2 \rfloor} (\lfloor \frac s 2 \rfloor)!} \frac {\Gamma(\frac {n} 2) \Gamma(\frac {k + s + 1} 2 - \lfloor \frac s 2 \rfloor) \Gamma(\frac {n - k} 2 + \lfloor \frac s 2 \rfloor)} {\Gamma(\frac {n + s + 1} 2) \Gamma(\frac {k} 2) \Gamma(\frac{n - k}{2})}.
      \end{align*}
    \end{corollary}

    The second special case concerns the tensorial curvature measures $\TenCM{k - 1}{r}{s}{0}$. Although
		the result has been derived in a different way in our previous work \cite[Theorem~4.12]{HugWeis16a},
		we state it  and  derive it  as a special case of the present more general approach.

    \begin{corollary} \label{Cor_CF_j=k-1_l=0}
      Let $K \in \cK^{n}$, $\beta \in \cB(\R^{n})$, and $k, r, s \in \N_{0}$ with $1 < k < n$. Then
      \begin{align*}
        & \int_{\AOp(n, k)} \TenCM{k - 1}{r}{s}{0} (K \cap E, \beta \cap E) \, \mu_k(\intd E) =
        \sum_{m = 0}^{\lfloor \frac s 2 \rfloor} \kappa_{n, k}^{s, m} \, Q^{m} \TenCM{n - 1}{r}{s - 2m}{0} (K, \beta),
      \end{align*}
      where
      \begin{align*}
        \kappa_{n, k}^{s, m} : = \frac{k - 1}{n - 1} \frac{1}{(4\pi)^{m} m!} \frac {\Gamma(\frac {n} 2) \Gamma(\tfrac {k + s - 1} 2 - m) \Gamma(\tfrac {n - k} 2 + m)} {\Gamma(\frac {n + s - 1} 2) \Gamma(\frac {k} 2) \Gamma(\frac{n - k}{2})}
      \end{align*}
      if $m \neq \frac {s - 1} 2$, and
      \begin{equation*}
        \kappa_{n, k}^{s, \frac {s - 1} 2} : = \frac{k (n + s - 2)}{2(n - 1)} \frac{1}{(4\pi)^{\frac {s - 1} 2} \frac {s - 1} 2!} \frac {\Gamma(\frac {n} 2) \Gamma(\tfrac {n - k + s - 1} 2)} {\Gamma(\frac {n + s + 1} 2) \Gamma(\frac{n - k}{2})}.
      \end{equation*}
    \end{corollary}

    Finally, we state the remaining case where $k = 1$ (see also  \cite[Theorem~4.13]{HugWeis16a}).

    \begin{corollary} \label{Cor_CF_k=1_l=0}
      Let $K \in \cK^{n}$, $\beta \in \cB(\R^{n})$, and $r, s \in \N_{0}$. Then
      \begin{align*}
        & \int_{\AOp(n, 1)} \TenCM{0}{r}{s}{0} (K \cap E, \beta \cap E) \, \mu_1( \intd E) \\
        & \qquad = \frac{\Gamma(\tfrac {s} 2 - \lfloor \frac s 2 \rfloor + 1)}{\sqrt\pi (4\pi)^{\lfloor \frac s 2 \rfloor} \lfloor \frac s 2 \rfloor!} \frac {\Gamma(\frac {n} 2) \Gamma(\tfrac {n + 1} 2 + \lfloor \frac s 2 \rfloor)} {\Gamma(\frac{n + 1}{2}) \Gamma(\frac {n + s + 1} 2)} \, Q^{\lfloor \frac s 2 \rfloor} \TenCM{n - 1}{r}{s - 2\lfloor \frac s 2 \rfloor}{0} (K, \beta).
      \end{align*}
    \end{corollary}

    Comparing Corollary \ref{Cor_CF_j=k-1_l=0} and Corollary \ref{Cor_CF_k=1_l=0} to the corresponding results in  \cite{HugWeis16a}, it should be observed that the normalization of the tensorial measures in \cite{HugWeis16a} is different from the current normalization (although the measures are denoted in the same way).

\section{The proofs of the main results}
\label{sec:4}

    In this section, we prove the Crofton formulae which have been stated in Section \ref{sec:3}.

\subsection{The kinematic formula for generalized tensorial curvature measures}

    The proof of the Crofton formulae uses the connection to the  corresponding (more general) kinematic formulae.
    For the classical scalar-valued curvature measures this connection is well known.
    For easier reference, we state the required kinematic formula, which has recently been  proved in \cite[Theorem 1]{HugWeis16b}.
    To state the result, we write $\GOp_n$ for the rigid motion group of $\R^n$ and denote by $\mu$ the Haar measure on $\GOp_n$ with the usual normalization (see \cite{HugWeis16a}, \cite[p.~586]{SchnWeil08}).

    \begin{theorem}[Kinematic formula \cite{HugWeis16b}] \label{Thm_KinemForm}
      For $P, P' \in \cP^n$, $\beta, \beta' \in \cB(\R^n)$, $j, l, r, s \in \N_{0}$ with $j \leq n$, and  $l=0$ if $j=0$,
      \begin{align*}
        & \int_{\GOp_n} \TenCM{j}{r}{s}{l} (P \cap g P', \beta \cap g \beta') \, \mu( \intd g) \nonumber \\
        & \qquad = \sum_{p = j}^{n} \sum_{m = 0}^{\lfloor \frac s 2 \rfloor} \sum_{i = 0}^{m} d_{n, j, n - p + j}^{s, l, i, m} \, Q^{m - i} \TenCM{p}{r}{s - 2m}{l + i} (P, \beta) \CM{n - p + j}(P', \beta'),
      \end{align*}
      where
      \begin{align*}
        d_{n, j, n - p + j}^{s, l, i, m} & = \frac{(-1)^{i}}{(4\pi)^{m} m!} \frac{\binom{m}{i}}{\pi^i} \frac {(i + l - 2)!} {(l - 2)!} \frac {\Gamma(\frac {n - p + j + 1} 2) \Gamma(\frac {p + 1} 2)} {\Gamma(\frac {n + 1} 2) \Gamma(\frac {j + 1} 2)} \\
        & \qquad \times \frac {\Gamma(\frac {p} 2 + 1)} {\Gamma(\frac {p + s} 2 + 1)} \frac {\Gamma(\frac {j + s} 2 - m + 1)} {\Gamma(\frac {j} 2 + 1)} \frac {\Gamma(\frac {p - j} 2 + m)} {\Gamma(\frac{p - j}{2})}.
      \end{align*}
    \end{theorem}

    In the formulation of Theorem \ref{Thm_KinemForm}, we changed the order of the coefficients slightly as compared to the original work (see \cite[Theorem 1]{HugWeis16b}), as we have $d_{n, j, n - p + j}^{s, l, i, m} = c_{n, j, p}^{s, l, i, m}$. This is done in order to shorten the representation of the Crofton formulae. Furthermore, since $\TenCM{n}{r}{\tilde s}{l}$ vanishes for $\tilde s \neq 0$ and the functionals $Q^{\frac{s}{2} - i} \TenCM{n}{r}{0}{l + i}$, $i \in \{ 0, \ldots, \frac{s}{2} \}$, can be combined, we can redefine
    \begin{align*}
      d_{n, j, j}^{s, l, i, m} & := \1 \{ s \text{ even}, m = i = \tfrac{s}{2} \} \frac{1}{(2\pi)^{s}
      (\frac{s}{2})!} \frac{\Gamma(\frac{n - j + s}{2})}{\Gamma(\frac{n - j}{2})};
    \end{align*}
    for further details see the remark after Theorem 1 in \cite{HugWeis16b}.
    In particular, we interpret the ratio $(i+l-2)!/(l-2)!$  as a ratio of Gamma functions for $l\in\{0,1\}$ (see also the comments after Theorem \ref{Thm_CF_j<k}).

    \subsection{The proofs}

    We prove both, Theorem \ref{Thm_CF_j=k} and Theorem \ref{Thm_CF_j<k}, at once using the kinematic formula for generalized tensorial curvature measures deduced in \cite{HugWeis16b} and restated in the last section as Theorem \ref{Thm_KinemForm}.

    \begin{proof}[Proof of Theorem \ref{Thm_CF_j=k} and Theorem \ref{Thm_CF_j<k}]
      Let $P \in \cP^n$ and $\beta \in \cB(\R^n)$. First, we   prove the identity
      \begin{align} \label{Form_Croft_2}
        J := \int_{\AOp(n, k)} \TenCM{j}{r}{s}{l} (P \cap E, \beta) \, \mu_k( \intd E) & = \int_{\GOp_n} \TenCM{j}{r}{s}{l} (P \cap g E_k, \beta \cap g \alpha) \, \mu( \intd g)
      \end{align}
      for an arbitrary (but fixed) $E_k \in \GOp(n, k)$ and $\alpha \in \cB(E_k)$ with $\cH^k(\alpha) = 1$. This is shown as follows. Using \eqref{decomp}, we obtain
      \begin{align*}
        J = \int _{\SO(n)} \int_{E_k^\perp} \int_{\R^n} \1_\beta(x) \, \TenCM{j}{r}{s}{l} (P \cap \rho (E_k + t_1), \intd x) \, \cH^{n - k}( \intd t_1) \, \nu(\intd \rho).
      \end{align*}
      For $t_1 \in E_k^\perp$ and $x \in \rho(E_k + t_1)$ we have
      \begin{align*}
        x \in \rho(\alpha + t_1 + t_2) \Leftrightarrow t_2 \in -\alpha + \rho^{-1} x- t_1,
      \end{align*}
      for all $t_2 \in E_k$. Moreover, $-\alpha + \rho^{-1} x - t_1 \subset E_k$, since $\alpha \subset E_k$ and $x \in \rho(E_k + t_1)$ yields $\rho^{-1} x- t_1 \in E_k$. Thus, we get
      \begin{align*}
        \cH^k \left( \left\{ t_2 \in E_k: x \in \rho(\alpha + t_1 + t_2) \right\} \right) = \cH^k(-\alpha + \rho^{-1} x- t_1) = \cH^k(\alpha) = 1,
      \end{align*}
      and hence we have
      \begin{align*}
        J & = \int _{\SO(n)} \int_{E_k^\perp} \int_{\R^n}  \1_\beta(x)\int_{E_k} \1\{ x \in \rho (\alpha + t_1 + t_2)\} \, \cH^k(\intd t_2) \\
        & \qquad \times \TenCM{j}{r}{s}{l} (P \cap \rho (E_k + t_1), \intd x) \, \cH^{n - k}( \intd t_1) \, \nu(\intd \rho) \\
        & = \int _{\SO(n)} \int_{E_k^\perp} \int_{E_k} \int_{\R^n} \1_{\beta \cap \rho (\alpha + t_1 + t_2)}(x) \, \TenCM{j}{r}{s}{l} (P \cap \rho (E_k + t_1 + t_2), \intd x) \, \\
        & \qquad \times \cH^k(\intd t_2) \, \cH^{n - k}( \intd t_1) \, \nu(\intd \rho).
      \end{align*}
      Finally, Fubini's theorem yields
      \begin{align*}
        J & = \int _{\SO(n)} \int_{\R^n} \TenCM{j}{r}{s}{l}(P \cap \rho(E_k + t), \beta \cap \rho(\alpha + t)) \, \cH^{n}(\intd t) \, \nu(\intd \rho) \\
        & = \int_{\GOp_n} \TenCM{j}{r}{s}{l} (P \cap g E_k, \beta \cap g \alpha) \, \mu( \intd g),
      \end{align*}
      which concludes the proof of (\ref{Form_Croft_2}).

      Let $\alpha \in \cB(\R^{n})$ be compact with $\alpha \subset E_{k}$ and $\cH^k(\alpha) = 1$.
      Then choose $P' \in \cP^n$ with $P' \subset E_k$ and $\alpha \subset \relint P'$, such that the following holds, for all $g \in \GOp_{n}$:
      If $g^{-1} P \cap \alpha \neq \emptyset$, then $g^{-1}P \cap E_{k} \subset P'$.
      Hence, if $P \cap g \alpha \neq \emptyset$, then $P \cap g E_k = P \cap g P'$.
      Thus we obtain
      \begin{align*}
        J & = \int _{\GOp_n} \TenCM{j}{r}{s}{l} (P \cap g P', \beta \cap g \alpha) \, \mu(\intd g),
      \end{align*}
      and thus, by Theorem \ref{Thm_KinemForm}
      \begin{align*}
        J & = \sum_{p = j}^{n} \sum_{m = 0}^{\lfloor \frac s 2 \rfloor} \sum_{i = 0}^{m} d_{n, j, n - p + j}^{s, l, i, m} \, Q^{m - i} \TenCM{p}{r}{s - 2m}{l + i} (P, \beta) \CM{n - p + j}(P', \alpha).
      \end{align*}
      Hence, if $k = j$ we get
      \begin{align*}
        J & = \1 \{ s \text{ even} \} \frac{1}{(2\pi)^{s} \left(\frac{s}{2}\right)!} \frac{\Gamma(\frac{n - k + s}{2})}{\Gamma(\frac{n - k}{2})} \, \TenCM{n}{r}{0}{\frac s 2 + l}(P, \beta)  \underbrace{\CM{k}(P', \alpha)}_{ = \cH^k(\alpha) = 1} \\
        & = \1 \{ s \text{ even} \} \frac{1}{(2\pi)^{s} \frac{s}{2}!} \frac{\Gamma(\frac{n - k + s}{2})}{\Gamma(\frac{n - k}{2})} \, \TenCM{n}{r}{0}{\frac s 2 + l}(P, \beta),
      \end{align*}
      and  for $j<k$ we get
      \begin{align*}
        J & = \sum_{m = 0}^{\lfloor \frac s 2 \rfloor} \sum_{i = 0}^{m} d_{n, j, k}^{s, l, i, m} \, Q^{m - i} \TenCM{n - k + j}{r}{s - 2m}{l + i} (P, \beta) \underbrace{\CM{k}(P',\alpha)}_{ = \cH^k(\alpha) = 1 } \\
        & = \sum_{m = 0}^{\lfloor \frac s 2 \rfloor} \sum_{i = 0}^{m} d_{n, j, k}^{s, l, i, m} \, Q^{m - i} \TenCM{n - k + j}{r}{s - 2m}{l + i} (P, \beta),
      \end{align*}
      since $\CM{q}(P',\alpha) = 0$ for $q \neq k$.
    \end{proof}

    Next, we prove Corollary \ref{Cor_CF_j=k-1_l=1} and  Corollary \ref{Cor_CF_j=k-1_l=1Alt}, which are derived from Theorem \ref{Thm_CF_j<k_l=1}.
    The first follows immediately, whereas the second  subsequently is obtained by an application of \eqref{FormTenCMRel}.

    \begin{proof}[Proof of Corollary \ref{Cor_CF_j=k-1_l=1} and  Corollary \ref{Cor_CF_j=k-1_l=1Alt}]
      In both cases, we denote the integral we are interested in by $I$. First, we consider
			Corollary \ref{Cor_CF_j=k-1_l=1} and hence $l=1$. In this case,
       Theorem \ref{Thm_CF_j<k_l=1} yields
      \begin{align*}
        I = \sum_{m = 0}^{\lfloor \frac s 2 \rfloor} \iota_{n, k}^{s, m} \, Q^{m} \TenCM{n - 1}{r}{s - 2m}{1} (K, \beta),
      \end{align*}
      where
      \begin{align*}
        \iota_{n, k}^{s, m} & : = d_{n, k - 1, k}^{s, 1, 0, m} = \frac{1}{(4\pi)^{m} m!} \frac {\Gamma(\frac {n} 2) \Gamma(\frac {k + s + 1} 2 - m) \Gamma(\frac {n - k} 2 + m)} {\Gamma(\frac {n + s + 1} 2) \Gamma(\frac {k} 2) \Gamma(\frac{n - k}{2})},
      \end{align*}
      which already proves Corollary \ref{Cor_CF_j=k-1_l=1}.
			
      Next, we turn to the proof of Corollary \ref{Cor_CF_j=k-1_l=1Alt}. From
			Corollary \ref{Cor_CF_j=k-1_l=1} and \eqref{FormTenCMRel},
			we conclude that
      \begin{align*}
        I & = \frac{2\pi}{n - 1} \sum_{m = 0}^{\lfloor \frac s 2 \rfloor} \iota_{n, k}^{s, m} Q^{m + 1} \TenCM{n - 1}{r}{s - 2m}{0} (K, \beta) - 2\pi(s - 2m + 2) \iota_{n, k}^{s, m} \, Q^{m} \TenCM{n - 1}{r}{s - 2m + 2}{0} (K, \beta) \\
        & = \frac{2\pi}{n - 1} \sum_{m = 1}^{\lfloor \frac s 2 \rfloor + 1} \iota_{n, k}^{s, m - 1} \, Q^{m} \TenCM{n - 1}{r}{s - 2m + 2}{0} (K, \beta) \\
        & \qquad - \frac{2\pi}{n - 1} \sum_{m = 0}^{\lfloor \frac s 2 \rfloor} 2\pi(s - 2m + 2) \iota_{n, k}^{s, m} \, Q^{m} \TenCM{n - 1}{r}{s - 2m + 2}{0} (K, \beta) \\
        & = \sum_{m = 1}^{\lfloor \frac s 2 \rfloor} \frac{2\pi}{n - 1} \left( \iota_{n, k}^{s, m - 1} - 2\pi(s - 2m + 2) \iota_{n, k}^{s, m} \right) \, Q^{m} \TenCM{n - 1}{r}{s - 2m + 2}{0} (K, \beta) \\
        & \qquad + \frac{2\pi}{n - 1} \iota_{n, k}^{s, \lfloor \frac s 2 \rfloor} \, Q^{\lfloor \frac s 2 \rfloor + 1} \TenCM{n - 1}{r}{s - 2\lfloor \frac s 2 \rfloor}{0} (K, \beta) - \frac{4\pi^{2}(s + 2)}{n - 1} \iota_{n, k}^{s, 0} \, \TenCM{n - 1}{r}{s + 2}{0} (K, \beta).
      \end{align*}
      Denoting the coefficients by $\lambda_{n, k}^{s, m}$, we obtain for $m \in \{ 1, \ldots, \lfloor \frac s 2 \rfloor \}$
      \begin{align*}
        \lambda_{n, k}^{s, m} & = \frac{\pi}{(n - 1)(4\pi)^{m - 1} m!} \frac {\Gamma(\frac {n} 2)\Gamma(\frac {k + s + 1} 2 - m) \Gamma(\tfrac {n - k} 2 + m - 1)} {\Gamma(\frac {n + s + 1} 2) \Gamma(\frac {k} 2) \Gamma(\frac{n - k}{2})} \\
        & \qquad \times \left( 2 m (\tfrac {k + s + 1} 2 - m) - (s - 2m + 2) (\tfrac {n - k} 2 + m - 1) \right),
      \end{align*}
      and
      \begin{align*}
        \lambda_{n, k}^{s, 0} & = - \frac{4\pi^{2}(s + 2)}{n - 1} \iota_{n, k}^{s, 0} = - \frac{4\pi^{2}(s + 2)}{n - 1} \frac {\Gamma(\frac {n} 2) \Gamma(\frac {k + s + 1} 2)} {\Gamma(\frac {n + s + 1} 2) \Gamma(\frac {k} 2)}, \\
        \lambda_{n, k}^{s, \lfloor \frac s 2 \rfloor + 1}& = \frac{2\pi}{n - 1} \iota_{n, k}^{s, \lfloor \frac s 2 \rfloor} \\
        & = \frac{2 \pi}{(n - 1)(4\pi)^{\lfloor \frac s 2 \rfloor} (\lfloor \frac s 2 \rfloor)!} \frac {\Gamma(\frac {n} 2) \Gamma(\frac {k + s + 1} 2 - \lfloor \frac s 2 \rfloor) \Gamma(\frac {n - k} 2 + \lfloor \frac s 2 \rfloor)} {\Gamma(\frac {n + s + 1} 2) \Gamma(\frac {k} 2) \Gamma(\frac{n - k}{2})},
      \end{align*}
      where $\lambda_{n, k}^{s, 0}$ is defined according to the general definition, but $\lambda_{n, k}^{s, \lfloor \frac s 2 \rfloor + 1}$ differs slightly for odd~$s$.
    \end{proof}

    Finally, we prove Corollary \ref{Cor_CF_j=k-1_l=0} and Corollary \ref{Cor_CF_k=1_l=0}, which are derived from Theorem~\ref{Thm_CF_j<k_l=0}.

    \begin{proof}[Proof of Corollary \ref{Cor_CF_j=k-1_l=0} and Corollary \ref{Cor_CF_k=1_l=0}]
        We denote the integral we are interested in by $I$ and establish both corollaries simultaneously.
				Theorem \ref{Thm_CF_j<k_l=0} yields
        \begin{align*}
          I = \sum_{m = 0}^{\lfloor \frac s 2 \rfloor} d_{n, k - 1, k}^{s, 0, 0, m} \, Q^{m} \TenCM{n - 1}{r}{s - 2m}{0} (K, \beta) + \sum_{m = 1}^{\lfloor \frac s 2 \rfloor} d_{n, k - 1, k}^{s, 0, 1, m} \, Q^{m - 1} \TenCM{n - 1}{r}{s - 2m}{1} (K, \beta),
        \end{align*}
        where
        \begin{align*}
          d_{n, k - 1, k}^{s, 0, i, m} & : = \frac{1}{4^{m} (m - i)!} \frac{1}{\pi^{i + m}} \frac {\Gamma(\frac {n} 2)} {\Gamma(\frac {n + s + 1} 2) \Gamma(\frac {k} 2) \Gamma(\frac{n - k}{2})}  \Gamma(\tfrac {k + s + 1} 2 - m) \Gamma(\tfrac {n - k} 2 + m).
        \end{align*}
        From \eqref{FormTenCMRel} we obtain that
        \begin{align*}
          I & = \sum_{m = 0}^{\lfloor \frac s 2 \rfloor} d_{n, k - 1, k}^{s, 0, 0, m} \, Q^{m} \TenCM{n - 1}{r}{s - 2m}{0} (K, \beta) + \frac{2\pi}{n - 1} \sum_{m = 1}^{\lfloor \frac s 2 \rfloor} d_{n, k - 1, k}^{s, 0, 1, m} \, Q^{m} \TenCM{n - 1}{r}{s - 2m}{0} (K, \beta) \\
          & \qquad - \frac{4\pi^{2}}{n - 1} \sum_{m = 1}^{\lfloor \frac s 2 \rfloor} d_{n, k - 1, k}^{s, 0, 1, m} (s - 2m + 2) \, Q^{m - 1} \TenCM{n - 1}{r}{s - 2m + 2}{0} (K, \beta),
        \end{align*}
       where we used that $d_{n, k - 1, k}^{s, 0, 1, 0} = 0$. This can be rewritten in the form
							\begin{align*}
          I & = \sum_{m = 0}^{\lfloor \frac s 2 \rfloor} \left( d_{n, k - 1, k}^{s, 0, 0, m} + \frac{2\pi}{n - 1} d_{n, k - 1, k}^{s, 0, 1, m} \right) Q^{m} \TenCM{n - 1}{r}{s - 2m}{0} (K, \beta) \\
          & \qquad - \frac{4\pi^{2}}{n - 1} \sum_{m = 0}^{\lfloor \frac s 2 \rfloor - 1} d_{n, k - 1, k}^{s, 0, 1, m + 1} (s - 2m) \, Q^{m} \TenCM{n - 1}{r}{s - 2m}{0} (K, \beta) \\
          & = \sum_{m = 0}^{\lfloor \frac s 2 \rfloor - 1} \left( d_{n, k - 1, k}^{s, 0, 0, m} + \frac{2\pi}{n - 1} d_{n, k - 1, k}^{s, 0, 1, m} - \frac{4\pi^{2} (s - 2m)}{n - 1} d_{n, k - 1, k}^{s, 0, 1, m + 1} \right) Q^{m} \TenCM{n - 1}{r}{s - 2m}{0} (K, \beta) \\
          & \qquad + \left( d_{n, k - 1, k}^{s, 0, 0, \lfloor \frac s 2 \rfloor} + \frac{2\pi}{n - 1} d_{n, k - 1, k}^{s, 0, 1, \lfloor \frac s 2 \rfloor} \right) Q^{\lfloor \frac s 2 \rfloor} \TenCM{n - 1}{r}{s - 2\lfloor \frac s 2 \rfloor}{0} (K, \beta).
        \end{align*}
        Denoting the corresponding coefficients of the summand
        $Q^{m} \TenCM{n - 1}{r}{s - 2m}{0} (K, \beta)$ by $\kappa_{n, k}^{s, m}$, we obtain
        \begin{align}
          \kappa_{n, k}^{s, m} & = \left( 1 + \frac{2m}{n - 1} \right) \frac{1}{(4\pi)^{m} m!} \frac {\Gamma(\frac {n} 2) \Gamma(\tfrac {k + s + 1} 2 - m) \Gamma(\tfrac {n - k} 2 + m)} {\Gamma(\frac {n + s + 1} 2) \Gamma(\frac {k} 2) \Gamma(\frac{n - k}{2})} \nonumber \\
          & \qquad - \frac{s - 2m}{n - 1} \frac{1}{(4\pi)^{m} m!} \frac {\Gamma(\frac {n} 2) \Gamma(\tfrac {k + s - 1} 2 - m) \Gamma(\tfrac {n - k} 2 + m + 1)} {\Gamma(\frac {n + s + 1} 2) \Gamma(\frac {k} 2) \Gamma(\frac{n - k}{2})} \nonumber \allowdisplaybreaks\\
          & = \frac{1}{(4\pi)^{m} m!} \frac {\Gamma(\frac {n} 2) \Gamma(\tfrac {k + s - 1} 2 - m) \Gamma(\tfrac {n - k} 2 + m)} {\Gamma(\frac {n + s + 1} 2) \Gamma(\frac {k} 2) \Gamma(\frac{n - k}{2})} \nonumber \\
          & \qquad \times \underbrace{\left( \tfrac{n + 2m - 1}{n - 1} (\tfrac {k + s - 1} 2 - m) - \tfrac{s - 2m}{n - 1} (\tfrac {n - k} 2 + m) \right)}_{ = \frac{k - 1}{n - 1} \frac{n + s - 1}{2} } \nonumber \\
          & = \frac{k - 1}{n - 1} \frac{1}{(4\pi)^{m} m!} \frac {\Gamma(\frac {n} 2) \Gamma(\tfrac {k + s - 1} 2 - m) \Gamma(\tfrac {n - k} 2 + m)} {\Gamma(\frac {n + s - 1} 2) \Gamma(\frac {k} 2) \Gamma(\frac{n - k}{2})}, \label{Form_kappa}
        \end{align}
        for $m \in \{ 0, \ldots, \lfloor \frac s 2 \rfloor - 1 \}$. For $k = 1$, we immediately get $\kappa_{n, 1}^{s, m} = 0$ in these cases. Furthermore, we have
        \begin{align*}
          \kappa_{n, k}^{s, \lfloor \frac s 2 \rfloor} & = \left( 1 + \frac{2\lfloor \frac s 2 \rfloor}{n - 1} \right) \frac{1}{(4\pi)^{\lfloor \frac s 2 \rfloor} \lfloor \frac s 2 \rfloor!} \frac {\Gamma(\frac {n} 2) \Gamma(\tfrac {k + s + 1} 2 - \lfloor \frac s 2 \rfloor) \Gamma(\tfrac {n - k} 2 + \lfloor \frac s 2 \rfloor)} {\Gamma(\frac {n + s + 1} 2) \Gamma(\frac {k} 2) \Gamma(\frac{n - k}{2})}.
        \end{align*}
        If $s$ is even and $k > 1$, this coincides with \eqref{Form_kappa} for $m = \frac{s}{2}$. If $s$ is odd, we have
        \begin{align*}
          \kappa_{n, k}^{s, \frac {s - 1} 2} & = \frac{k (n + s - 2)}{2(n - 1)} \frac{1}{(4\pi)^{\frac {s - 1} 2} \frac {s - 1} 2!} \frac {\Gamma(\frac {n} 2) \Gamma(\tfrac {n - k + s - 1} 2)} {\Gamma(\frac {n + s + 1} 2) \Gamma(\frac{n - k}{2})}
        \end{align*}
        and thus the assertion of Corollary \ref{Cor_CF_j=k-1_l=0}. For $k = 1$, we obtain
        \begin{align*}
          \kappa_{n, 1}^{s, \lfloor \frac s 2 \rfloor} & = \frac{\Gamma(\tfrac {s} 2 - \lfloor \frac s 2 \rfloor + 1)}{\sqrt\pi (4\pi)^{\lfloor \frac s 2 \rfloor} \lfloor \frac s 2 \rfloor!} \frac {\Gamma(\frac {n} 2) \Gamma(\tfrac {n + 1} 2 + \lfloor \frac s 2 \rfloor)} {\Gamma(\frac{n + 1}{2}) \Gamma(\frac {n + s + 1} 2)}
        \end{align*}
        and thus the assertion of Corollary \ref{Cor_CF_k=1_l=0}.
    \end{proof}

\section{Linear independence of the generalized tensorial curvature measures}
\label{sec:5}

In this section, we show that the generalized tensorial curvature measures multiplied with powers of the metric tensor are linearly independent.
The proof of this result follows the argument for Theorem 3.1 in \cite{HugSchn14}.

\begin{theorem}\label{linindep}
	For $p \in \N_0$, the tensorial measure valued valuations $$Q^m \LokMinTen{j}{r}{s}{l}: \cP^{n} \times \cB(\R^{n}) \rightarrow \T^{p}$$ with $ m, r, s, l \in \N_0$ and
	$j \in \{ 0, \ldots, n  \}$, where $2m + 2l + r + s = p$,   but with $l = 0$ if $j \in \{ 0, n - 1 \}$ and with $s=l=0$ if $j=n$, are linearly independent.
\end{theorem}

In particular, Theorem \ref{linindep} shows that the Crofton formulae, which we stated in Section \ref{sec:3} and proved in Section \ref{sec:4}, cannot be simplified further, as there are no more linear dependences between the appearing functionals. A corresponding statement holds for the results from our preceding work~\cite{HugWeis16b}.

\begin{proof}
	Suppose that
	\begin{align} \label{Form_Indep_1}
		\sum_{\substack{j,  m, r, s, l  \\ 2m + 2l + r + s = p}} a_{j, m, r, s, l}^{(0)} Q^m \LokMinTen{j}{r}{s}{l} = 0
	\end{align}
	holds for some $a_{j,  m, r, s, l}^{(0)} \in \R$, where $a_{0,  m, r, s, l}^{(0)} = a_{n - 1,  m, r, s, l}^{(0)}  = 0$ if $l \neq 0$ and $a_{n , m, r, s, l}^{(0)}=0$ if $s\neq 0$ or $l\neq 0$. In the proof we will
	replace the constants $a_{j, m, r, s, l}^{(0)}$ by new constants $a_{j, m, r, s, l}^{( 1)}$ without keeping track of the precise relations, since it will be sufficient to know that
	$a_{j, m, r, s, l}^{(0)} = 0$ if and only if $a_{j, m, r, s, l}^{(1)} = 0$.

	For a fixed $j \in \{ 0, \ldots, n \}$, let $P \in \cP^n$ with $\interior P \neq \emptyset$, $F \in \cF_j(P)$,
	and $\beta \in \cB(\relint F)$. Then, if $j < n$ we obtain for the generalized tensorial curvature measures
	\begin{align*}
		\LokMinTen{j}{r}{s}{l}(P, \beta) & = c_{n, j, r, s,l} \sum_{G \in \cF_j(P)} Q(G)^l \int_{G \cap \beta} x^r \, \cH^j(\intd x) \int_{N(P, G) \cap \Sn} u^s \, \cH^{n - j - 1} (\intd u) \\
		& = c_{n, j, r, s,l} Q(F)^l \int_{\beta} x^r \, \cH^j(\intd x) \int_{N(P,F) \cap \Sn} u^s \, \cH^{n - j - 1}
		(\intd u),
	\end{align*}
	where $c_{n, j, r, s,l}>0$ is a constant, and $\LokMinTen{k}{r}{s}{l}(P, \beta) = 0$ for $k \neq j$. Moreover, we
	have
	$$
	\LokMinTen{n}{r}{0}{0}(P, \beta) =\frac{1}{r!}\int_\beta x^r\, \mathcal{H}^n(\intd x).
	$$
	Hence, from (\ref{Form_Indep_1}) it follows that
	\begin{align*}
		\sum_{\substack{ m, r, s, l  \\ 2m + 2l + r + s = p}} a_{j,  m, r, s, l}^{(1)} Q^m Q(F)^l \int_{\beta} x^r \, \cH^j(\intd x) \int_{N(P,F) \cap \Sn} u^s \, \cH^{n - j - 1} (\intd u) = 0,
	\end{align*}
	where for $j=n$ the spherical integral is omitted (also in the following).
	
	We may assume that $\int_\beta x^r \,\cH^j(\intd x) \neq 0$ (otherwise, we consider a translate of $P$ and $\beta$). If we repeat the above calculations with multiples of $P$ and $\beta$, a comparison of the degrees of homogeneity yields,  for every $r \in \N_0$, that
	\begin{align*}
		\sum_{\substack{ m, s, l  \\ 2m + 2l + s = p - r}} a_{j,  m, r, s, l}^{(1)} Q^m Q(F)^l \int_{\beta} x^r \,
		\cH^j(\intd x) \int_{N(P,F) \cap \Sn} u^s \, \cH^{n - j - 1} (\intd u) = 0.
	\end{align*}
	Hence, due to the lack of zero divisors in the tensor algebra $\T$, we obtain
	\begin{align} \label{Form_Indep_2}
		\sum_{\substack{ m, s, l  \\ 2m + 2l + s = p - r}} a_{j,  m, r, s, l}^{(1)} Q^m Q(F)^l \int_{N(P,F) \cap \Sn} u^s \, \cH^{n - j - 1} (\intd u) = 0.
	\end{align}
This shows that, in the case of $j=n$ (where the spherical integral is omitted), we have $a_{n,  m, r, s, l}^{(1)}=0$ also for $s=l=0$.
Hence, in the following we may assume that $j<n$.

Let $L\in \GOp(n, j)$, $j<n$, and $u_0\in L^\perp\cap\mathbb{S}^{n-1}$. For $j\le n-2$, let
$u_0,u_1,\ldots,u_{n-j-1}$ be an orthonormal
basis of $L^\perp$. In this case, we define the (pointed) polyhedral cone
$C(u_0,\tau):=\text{pos}\{u_0\pm \tau\, u_1,\ldots,u_0\pm \tau\, u_{n-j-1}\}\subset L^\perp$ for $\tau\in (0,1)$.
Then, for any $v\in C(u_0,\tau)\cap \Sn$, we have $\langle v, u_0\rangle\ge 1/\sqrt{1+\tau^2}$, and therefore $\|u_0-v\|\le \sqrt{2}\tau$.
In fact,  any $v \in C(u_0,\tau)\cap \Sn$ can be written as $v = \frac{x}{\| x \| }$, where $x \in \conv\{ v_{1}^{\pm}, \ldots, v_{n - j - 1}^{\pm} \}$ with $v_{i}^{\pm} = \frac{u_0\pm \tau u_i}{\| u_0\pm \tau u_i \|} = \frac{u_0\pm \tau u_i}{\sqrt{1 + \tau^2}} \in \Sn$, $i \in \{ 1, \ldots, n - j - 1 \}$. Thus we have $x = \sum \lambda_{i}^{\epsilon} v_{i}^{\epsilon}$, where we sum
 over all $i \in \{ 1, \ldots, n - j - 1 \}$ and all $\epsilon \in \{+, -\}$, with $\sum \lambda_{i}^{\epsilon} = 1$ and $\lambda_{i}^{\epsilon} \ge 0$, $i \in \{ 1, \ldots, n - j - 1 \}$, $\epsilon \in \{+, -\}$.
This yields
\begin{align*}
  \langle v, u_{0} \rangle
  = \frac{1}{\| x \|} \langle \sum \lambda_{i}^{\epsilon} v_{i}^{\epsilon}, u_{0} \rangle
  = \frac{1}{ \sqrt{1 + \tau^2} \| x \|} \sum \lambda_{i}^{\epsilon}
  \geq \frac{1}{\sqrt{1 + \tau^2}},
\end{align*}
as $\| x \| \leq \sum \lambda_{i}^{\epsilon} \| v_{i}^{\epsilon} \| = 1$. This proves the assertion.
For $j=n-1$ we simply put $C(u_0,\tau):=\text{pos}\{u_0\}$.

Let $C(u_0,\tau)^\circ$ denote the polar cone of $C(u_0,\tau)$. Then $P:=C(u_0,\tau)^\circ\cap [-1,1]^n\in\mathcal{P}^n$
and $F:=L\cap[-1,1]^n\in\mathcal{F}_j(P)$ satisfy $N(P,F)=N(P,0)=C(u_0,\tau)$. With these choices, \eqref{Form_Indep_2} turns into
\begin{equation}\label{new2}
\sum_{\substack{ m, s, l  \\ 2m + 2l + s = p - r}} a_{j,  m, r, s, l}^{(1)}
Q^m Q(L)^l \int_{C(u_0,\tau) \cap \Sn} u^s \, \cH^{n - j - 1} (\intd u) = 0.
\end{equation}
Dividing \eqref{new2} by $\cH^{n - j - 1}(C(u_0,\tau) \cap \Sn)$ and passing to the limit as $\tau\to 0$, we get
\begin{align*} 
		\sum_{\substack{ m, s, l  \\ 2m + 2l + s = p - r}} a_{j,   m, r, s, l}^{(1)} Q^m Q(L)^l u_0^s = 0
	\end{align*}
	for any $u_0 \in L^\perp\cap \Sn$. Here we use that
\begin{align*}
&\left|\cH^{n - j - 1}(C(u_0,\tau) \cap \Sn)^{-1}\int_{C(u_0,\tau) \cap \Sn} u^s \, \cH^{n - j - 1} (\intd u)
-u_0^s\right|\\
&\qquad\le \max\{|u^s-u_0^s|:u\in C(u_0,\tau) \cap \Sn\}\to 0
\end{align*}
 as $\tau\to 0$.
	
	The rest of the proof follows similarly as in the proof of \cite[Theorem 3.1]{HugSchn14}.
\end{proof}

\end{document}